\documentclass[a4paper,10pt]{amsart}
\usepackage[utf8]{inputenc}
\usepackage[inner=2.5cm,outer=2.5cm,top=3cm,headsep=1cm, footskip=1cm,bottom=3cm]{geometry}
\usepackage[english]{babel}
\usepackage{amsfonts,latexsym,rawfonts,amsmath,amssymb,amsthm,amscd}
\usepackage{enumerate}
\usepackage{stackrel}
\usepackage{graphicx}
\usepackage{array} 
\input xy
\xyoption{all}
\usepackage{color}

\newtheorem{prop}{Proposition}[section]
\newtheorem{teo}[prop]{Theorem}
\newtheorem{lem}[prop]{Lemma}
\newtheorem{cor}[prop]{Corollary}

\theoremstyle{definition}
\newtheorem{nada}[prop]{}
\newtheorem{defi}[prop]{Definition}
\newtheorem{example}[prop]{Example}
\newtheorem{examples}[prop]{Examples}
\newtheorem{rmk}[prop]{Remark}

\theoremstyle{theorem}

\def\Ho{\mathrm{Ho}}
\def\Hom{\mathrm{Hom}}
\def\Dec{\mathrm{Dec}}

\def\Ker{\mathrm{Ker }}

\newcommand{\pb}{\ar@{}[dr]|{\mbox{\LARGE{$\lrcorner$}}}}

\newcommand{\xra}[1]{\xrightarrow{#1}}

\newcommand{\Cx}[1]{\mathbf{C}^+(#1)}
\newcommand{\FCx}[1]{\mathbf{C}^+(\mathbf{F}#1)}

\newcommand{\dga}[2]{\mathsf{DGA}^{#1}({#2})}

\newcommand{\Sch}[1]{\mathbf{Sch}(#1)}
\newcommand{\Sm}[1]{\mathbf{Sm}(#1)}
\newcommand{\Man}[1]{\mathbf{Man}(#1)}
\newcommand{\An}[1]{\mathbf{An}(#1)}

\newcommand{\Manc}[1]{\widehat{\mathbf{Man}}(#1)}
\newcommand{\Anc}[1]{\widehat{\mathbf{An}}(#1)}

\newcommand{\Rmod}{A\text{\normalfont{-mod}}}

\newcommand{\wt}{\widetilde}
\newcommand{\lra}{\longrightarrow}

\newcommand{\CC}{\mathbb{C}}

\newcommand{\PP}{\mathbb{P}}
\newcommand{\QQ}{\mathbb{Q}}
\newcommand{\RR}{\mathbb{R}}

\newcommand{\ZZ}{\mathbb{Z}}

\newcommand{\Aa}{\mathcal{A}}

\newcommand{\Cc}{\mathcal{C}}
\newcommand{\Dd}{\mathcal{D}}
\newcommand{\Ee}{\mathcal{E}}

\newcommand{\Mm}{\mathcal{M}}

\newcommand{\Ss}{\mathcal{S}}
\newcommand{\Ww}{\mathcal{W}}

\newcommand{\Rr}{\mathbf{R}}
\newcommand{\kk}{\mathbf{k}}

\title[Weight filtration on the cohomology of complex analytic spaces]{Weight filtration on the cohomology\\ of complex analytic spaces}

\author[J. Cirici]{Joana Cirici}
\address[J. Cirici]{
Fachbereich Mathematik und Informatik\\
Freie Universit\"{a}t Berlin\\  Arnimallee 3\\ 
14195 Berlin}
\email{jcirici@math.fu-berlin.de}

\author[F. Guill\'{e}n]{Francisco Guill\'{e}n}
\address[F. Guill\'{e}n]{Departament
d'\`{A}lgebra i Geometria\\  Universitat de Barcelona\\ Gran Via 585,
08007 Barcelona}
\email{fguillen@ub.edu}

\subjclass[2010]{32C18 (primary), 32S35 (secondary).}
\keywords{weight filtration, cohomological descent, cubical hyperresolutions, mixed Hodge theory, analytic spaces}
\thanks{The first-named author 
wants to acknowledge financial support from the Marie Curie Action through PCOFUND-GA-2010-267228 and the DFG through project SFB 647.}

\begin{document}

\begin{abstract}
We extend Deligne's weight filtration to the integer cohomology of complex analytic spaces 
(endowed with an equivalence class of compactifications). In general, the weight filtration that we obtain is
not part of a mixed Hodge structure. 
Our purely geometric proof is based on cubical descent
for resolution of singularities and Poincar\'{e}-Verdier duality. 
Using similar techniques, we introduce the singularity filtration on the cohomology of compactificable analytic spaces.
This is a new and natural analytic invariant which does not depend on the equivalence class of compactifications and is related to the weight filtration.
\end{abstract}
\maketitle
\section*{Introduction}
The weight filtration was introduced by Deligne \cite{DeHII}, \cite{DeHIII} following Grothendieck's yoga of weights,
and as the key ingredient of mixed Hodge theory.
This is an increasing filtration defined functorially on the rational cohomology of every complex algebraic variety,
expressing the way in which its cohomology is related to 
cohomologies of smooth projective varieties. In Deligne's approach, the weight filtration on the cohomology
of a singular complex algebraic variety $X$, supposed compact to simplify, is the
filtration induced by a smooth hypercovering $X_\bullet\to X$ of $X$. Indeed,
the induced spectral
sequence 
$$E_1^{p,q}(X;A)=H^q(X_p;A)\Rightarrow H^{p+q}(X;A)$$
defines a filtration on $H^n(X;A)$ for any given coefficient ring $A$. 
Using Hodge theory, Deligne proved that when $A=\QQ$, the above spectral sequence
degenerates at the second stage and that the 
filtration on the rational cohomology is well-defined and functorial.

In \cite{GiSo}, Gillet and Soul\'{e} gave
an alternative proof of the well-definedness of the weight filtration 
using smooth hypercoverings and algebraic K-theory.
Their more geometric approach allowed them to
obtain the result with integer coefficients, at least for compact support cohomology.
For a general coefficient ring $A$, the above spectral sequence
does not necessarily degenerate at the second stage. However, they proved that, from the second stage onwards,
the corresponding spectral sequence for cohomology with compact supports is a well-defined algebraic invariant of the variety.
An analogous construction yielded a weight filtration on algebraic K-theory with compact supports (see also the work of Pascual-Rubi\'{o} \cite{PR}).

Based on cubical hyperresolutions (see \cite{GNPP}),
Guill\'{e}n and Navarro-Aznar \cite{GN} developed a general descent theory which
allows one to extend contravariant
functors compatible with elementary acyclic squares (smooth blow-ups) on the category of smooth
schemes, to functors on the category of all schemes in such a way that the extended
functor is compatible with acyclic squares (abstract blow-ups).

Totaro observed in his ICM lecture \cite{To} that using the results on cohomological descent
of \cite{GN}, the weight filtration is well-defined on the cohomology with compact supports of any
complex or real analytic space with a given equivalence class of compactifications, for which
Hironaka's resolution of singularities holds.
Following this idea, and using Poincar\'{e} duality for real manifolds with
$\ZZ_2$-coefficients, McCrory and Parusinski obtained the weight filtration for
real algebraic varieties on Borel-Moore $\ZZ_2$-homology in \cite{MCPI}, and on compactly supported
$\ZZ_2$-homology in \cite{MCPII}.

Let $X$ be a compactificable complex analytic space.
Every cubical hyperresolution $X_\bullet \to X$ induces a spectral sequence
$$E_1^{p,q}(X;A)=\bigoplus_{|\alpha|=p}H^q(X_\alpha;A)\Rightarrow H^{p+q}(X;A)$$
converging to a filtration of $H^n(X;A)$, which we call \textit{singularity filtration}. In this note we prove 
that from the $E_2$-term onwards, this spectral sequence does not depend
on the choice of the hyperresolution (Theorem $\ref{singan}$).
In the same vein, in Theorem $\ref{weightan}$ we obtain a generalization of the weight filtration
on the cohomology $H^n(X;A)$, when $X$ is endowed with an
equivalence class of compactifications (see Definition $\ref{equivclass}$).
In particular, if $X$ is a complex algebraic variety, the  weight filtration is well-defined on $H^n(X;A)$.
For smooth manifolds, the singularity filtration is trivial, while for compact analytic spaces, it
coincides with the weight filtration.

To obtain these results we give
an analytic version of the extension criterion of functors
of \cite{GN} (see Theorem $\ref{descens1an}$). 
The analytic setting differs from the algebraic setting appearing in loc.cit., mainly
due to the weaker formulation of Chow-Hironaka's Lemma and certain finiteness issues,
and it may find applications to study other topological invariants.
For instance, it should allow one to define a Hodge and a weight
filtration on the rational homotopy type of complex analytic spaces,
extending the filtrations obtained by Morgan \cite{Mo}, to the analytic setting.
We will present this multiplicative theory elsewhere.

Our study of the singularity and the weight filtrations is also valid for the cohomology
with $\ZZ_2$ coefficients of real analytic spaces.
Other cohomology theories such as Borel-Moore homology or cohomology
with compact supports could also be studied using parallel techniques,
allowing a comparison with the quoted results of Gillet-Soul\'{e} and McCrory-Parusinski.
\\

Section $\ref{prelims}$ contains a brief exposition of the main results of \cite{GN}
on cohomological descent categories and the extension criterion of functors.

In Section $\ref{fcx}$ we show that the category of filtered complexes over an abelian category
admits a cohomological descent structure with respect to the class of weak equivalences given by $E_r$-quasi-isomorphisms:
morphisms of filtered complexes inducing a quasi-isomorphism 
at the $r$-stage of the associated spectral sequence (see Theorem $\ref{cohdescentfiltcomplexes}$).

In Section $\ref{ext_anal}$ we establish an extension criterion of functors from smooth to singular 
complex analytic spaces (Theorem $\ref{descens1an}$) as well as a relative version (Theorem $\ref{descens2an}$).

In Section $\ref{acyclicity}$ we study the behavior of the cohomology functor with respect to acyclic squares of analytic spaces.
We then study the Gysin complex of a smooth compactification $U\hookrightarrow X$ with $D=X-U$ a normal crossings divisor.

In Sections $\ref{sec_sing}$ and $\ref{sec_weightan}$ we use the results of the previous sections to
define the singularity and weight filtrations respectively, on the cohomology with coefficients in an arbitrary ring $A$,
of compactificable complex analytic spaces (Theorem $\ref{singan}$ and Theorem $\ref{weightan}$).

\section{Preliminaries}\label{prelims}
The extension criterion of functors of \cite{GN} is based on the assumption that the target category
is a \textit{cohomological descent category}, a variant of the triangulated categories of Verdier.
This is essentially a category $\Dd$ endowed with a saturated class of \textit{weak equivalences} $\Ee$,
and a \textit{simple functor} $\mathbf{s}$ sending every cubical codiagram of $\Dd$ to an object of $\Dd$ and
satisfying certain axioms analogous to those of the total complex of a double complex.
The simple functor can be viewed as the homotopy limit, and allows to define
realizable homotopy limits for diagrams indexed by finite categories (see \cite{Bea}).

The basic idea of the extension criterion of functors is that, given a functor compatible with smooth
blow-ups from the category of smooth schemes to a cohomological descent category, there
exists an extension to all schemes. Furthermore, the extension is essentially unique,
and is compatible with general blow-ups.

Let us first recall some features of cubical codiagrams and cohomological descent categories.
We refer to \cite{GN} for the precise definitions.

\begin{nada}
Given a set $\{0,\cdots,n\}$, with $n\geq 0$, the set of its non-empty parts, 
ordered by the inclusion, defines the category $\square_n$.
Likewise, any non-empty
finite set $S$ defines the category $\square_S$.
Every injective map $u:S\to T$ between non-empty finite sets
induces a functor $\square_u:\square_S\to\square_T$
defined by $\square_u(\alpha)=u(\alpha)$.
Denote by $\Pi$ 
the category whose objects are finite products of categories $\square_S$ and whose morphisms
are the functors associated with injective maps in each component.
\end{nada}

\begin{defi}
Let $\Dd$ be an arbitrary category. A \textit{cubical codiagram} of $\Dd$ is
a pair $(X,\square)$, where $\square$ is an object of $\Pi$ and $X$ is a functor $X:\square\to \Dd$.
A \textit{morphism} $(X,\square)\to (Y,\square')$ between cubical codiagrams is given by a pair $(a,\delta)$ where
$\delta:\square'\to \square$ is a morphism of $\Pi$ and $a:\delta^*X:=X\circ\delta \to Y$ is a natural transformation.
\end{defi}
Denote by $CoDiag_{\Pi}\Dd$ the category of cubical codiagrams of $\Dd$.

\begin{defi}
A \textit{cohomological descent category} is given by 
a cartesian category $\Dd$ provided with an initial object $1$, together with
a saturated class of morphisms $\Ee$ of $\Dd$ which is stable by products,
called \textit{weak equivalences}, and a contravariant functor 
$\mathbf{s}:CoDiag_{\,\Pi}\Dd\to \Dd$, called the \textit{simple functor}.
The data $(\Dd,\Ee,\mathbf{s})$ must satisfy the axioms of Definition 1.5.3 of \cite{GN}.
Objects weakly equivalent to the initial object $1$ are called \textit{acyclic}.
We shall denote by $\Ho(\Dd):=\Dd[\Ee^{-1}]$ the localized category of $\Dd$ with respect to the class of weak equivalences.
\end{defi}

The primary example of a cohomological descent category is given by the category of complexes $\Cx{\Aa}$ of an abelian category $\Aa$,
with the weak equivalences being quasi-isomorphisms and the simple
functor $\mathbf{s}$ defined via the total complex.

Let $\Dd$ be a category with initial object and a simple functor $\mathbf{s}:CoDiag_{\,\Pi}\Dd\to \Dd$.
In a large class of examples, a cohomological descent structure on $\Dd$ is given after lifting the class of
quasi-isomorphisms on $\Cx{\Aa}$ via a functor $\Psi:\Dd\lra \Cx{\Aa}$, provided that $\Psi$ is compatible
with the simple functor (see Proposition 1.5.12 of \cite{GN}).

\begin{examples} The following categories inherit a cohomological descent structure via functors with values in categories of complexes:
\begin{enumerate}[(1)]
\item The category $\mathrm{Top}$ of topological spaces, with the simple functor $\mathbf{s}$ given by the geometric realization of cubical diagrams,
via the functor $S_*:\mathrm{Top}\lra C_+(\ZZ{\mathrm{-mod}})$ of singular chains.
\item The category $\dga{}{\kk}$ of differential graded algebras over a field $\kk$ of characteristic 0,
with the Thom-Whitney simple functor $\mathbf{s}_{TW}$ (see \cite{Na}), via the functor 
$\dga{}{\kk}\lra \Cx{\mathrm{Vect}(\kk)}$ defined by forgetting the algebra structures.
\item The category $\mathbf{MHC}$ of mixed Hodge complexes with the cubical analog of Deligne's 
simple functor $\mathbf{s}_D$ (see \cite{DeHIII}),
via the forgetful functor $\mathbf{MHC}\lra \Cx{\mathrm{Vect}(\QQ)}$.
\end{enumerate}
\end{examples}
In the Section $\ref{fcx}$ we will lift the cohomological descent structure on $\Cx{\Aa}$ 
to a family of structures for filtered complexes, via the functor defined by the $r$-stage of 
the associated spectral sequences.

\begin{nada}[$\Phi$-rectified functors]
Let $\Dd$ be a cohomological descent category and let 
$\square\in \Pi$. Denote by $\Dd^\square:=Fun(\square,\Dd)$ the category of diagrams of type $\square$ in $\Dd$.
The simple functor induces a functor $\Ho(\Dd^\square)\to \Ho(\Dd)$.
In general, we are interested in cubical diagrams in $\Ho(\Dd)$
and we do not have a simple functor $\Ho(\Dd)^\square\to \Ho(\Dd)$.
The notion of \textit{$\Phi$-rectified functor} corresponds, roughly speaking, to functors
$F:\Cc\to \Ho(\Dd)$ which are defined on all cubical diagrams in the form
$F^\square:\Cc^\square\to \Ho(\Dd^\square)$,
so that we can take the composition $\Cc^\square\to \Ho(\Dd^\square)\to \Ho(\Dd)$
(see 1.6 of \cite{GN}).
\end{nada}

Denote
by $\Sch{\kk}$ the category of reduced schemes, that are separated and of finite type over a field $\kk$
of characteristic 0. Denote by $\Sm{\kk}$ the full subcategory of smooth schemes.

\begin{defi}\label{defelement1_alg}
A cartesian diagram of $\Sch{\kk}$
$$
\xymatrix{
\ar[d]_g\wt Y\ar[r]^j&\wt X\ar[d]^f\\
Y\ar[r]^i&X
}
$$
is said to be an \textit{acyclic square} if $i$ is a closed immersion, $f$ is proper and it induces an isomorphism
$\wt X-\wt Y\to X-Y$.
It is an \textit{elementary acyclic square} if, in addition, 
all the objects in the diagram are in $\Sm{\kk}$, and $f$ is the blow-up of $X$ along $Y$.
\end{defi}

\begin{teo}[\cite{GN}, Theorem. 2.1.5]\label{descens1alg}Let $\Dd$ be a cohomological descent category
and
$F:\Sm{\kk}\to \Ho(\Dd)$
a contravariant $\Phi$-rectified functor satisfying:
\begin{enumerate}
 \item [{(F1)}] $F(\emptyset)$ is the final object of $\Dd$ and $F(X\sqcup Y)\to F(X)\times F(Y)$ is an isomorphism.
 \item [{(F2)}] If $X_\bullet$ is an elementary acyclic square of $\Sm{\kk}$, then
$\mathbf{s}F(X_\bullet)$ is acyclic.
\end{enumerate}
Then there exists a contravariant $\Phi$-rectified
functor 
$F':\Sch{\kk}\lra \Ho(\Dd)$
such that:
\begin{enumerate}
\item [(1)]If $X$ is an object of $\Sm{\kk}$, then $F'(X)\cong F(X)$.
 \item [(2)]If $X_{\bullet}$ is an acyclic square of $\Sch{\kk}$, then $\mathbf{s}F'(X_{\bullet})$ is acyclic.
 \end{enumerate}
In addition, the functor $F'$ is essentially unique.
\end{teo}

The main applications of the extension criterion appearing in \cite{GN} 
concern the study of
the algebraic de Rham homotopy theory,
the Hodge filtration for complex analytic spaces and
the theory of Grothendieck motives (see Sections 3,4 and 5 of loc.cit.).
The extension criterion has been successfully applied on other algebraic situations, such as
the weight filtration in algebraic K-theory \cite{PR},
the mixed Hodge structures in rational homotopy \cite{Na}, \cite{CG1}, \cite{cirici}
or the weight filtration for real algebraic varieties \cite{MCPII}.

\section{Cohomological descent structures on the category of filtered complexes}\label{fcx}

In this section we show that the category of filtered complexes over an abelian category admits a cohomological descent structure, 
where the weak equivalences are given by $E_r$-quasi-isomorphisms.

Let $\Aa$ be an abelian category. Denote by
$\mathbf{F}\Aa$ the category of filtered objects of $\Aa$,
and by $\FCx{\Aa}$ the category of complexes over objects of $\mathbf{F}\Aa$.
\begin{defi}
Let $r\geq 0$ be an integer.
A morphism of filtered complexes $f:K\to L$ is called \textit{$E_r$-quasi-isomorphism} if
the induced morphism $E_r(f):E_r(K)\to E_r(L)$ of the associated spectral sequences is a quasi-isomorphism
(the map $E_{r+1}(f)$ is an isomorphism).
\end{defi}
Denote by $\Ee_r$ the class of $E_r$-quasi-isomorphisms.
\begin{defi}\label{rderived}
The \textit{$r$-derived category} of $\mathbf{F}\Aa$ is the localized category
$$\mathbf{D}^+_r(\mathbf{F}\Aa):=\FCx{\Aa}[\Ee_r^{-1}]$$
of complexes of $\mathbf{F}\Aa$
with respect to the class of $E_r$-quasi-isomorphisms.
\end{defi}
Let $s>r$ and consider the functor $E_{s}:\FCx{\Aa}\to \Cx{\Aa}$ defined by sending a filtered complex to
the $E_s$-stage of its associated spectral sequence.
Since it sends morphisms of $\Ee_r$ to isomorphisms, there is a functor
$E_{s}:\mathbf{D}^+_r(\mathbf{F}\Aa)\to \Cx{\Aa}$.

Deligne introduced the d\'{e}calage of a filtered complex and
proved that its associated spectral sequences are related by a shift of indexing.
This proves to be a key tool in the study of filtered complexes and their cohomological descent properties.
\begin{defi}\label{decalage_defi_algs}
The \textit{d\'{e}calage} $\Dec K$ of a filtered complex $K$ is the filtered complex defined by
$$(\Dec W)_pK^n=W_{p-n}K^n\cap d^{-1}(W_{p-n-1}K^{n-1}).$$
\end{defi}

\begin{prop}[\cite{DeHII}, Prop. 1.3.4] \label{deligne_decalage}
The canonical morphism
$E_0^{p,n-p}(\Dec K)\to E_1^{p+n,-p}(K)$ is a quasi-isomorphism.
For all $r>0$, the induced morphism
$E_r^{p,n-p}(\Dec K)\to E_{r+1}^{p+n,-p}(K)$
is an isomorphism. In particular $\Ee_{r}=\Dec^{-1}(\Ee_{r-1})$ for all $r>0$.
\end{prop}

\begin{defi}
The \textit{$r$-simple} of a codiagram of filtered
complexes $K=(K,W)^\bullet$ is the filtered complex 
$\mathbf{s}^r(K):=(\mathbf{s}(K),W(r))$
defined by
$$W(r)_p(\mathbf{s}(K))=\int_\alpha C^*(\Delta^{|\alpha|})\otimes W_{p+r|\alpha|}K^\alpha=
\bigoplus_{|\alpha|=0} W_pK^\alpha\oplus \bigoplus_{|\alpha|=1} W_{p+r}K^\alpha[-1]\oplus\cdots
$$
\end{defi}
Note that $\mathbf{s}^0$ and $\mathbf{s}^1$ correspond to the
filtered total complexes defined via the convolution with the trivial and the b\^{e}te filtrations respectively,
introduced by Deligne in \cite{DeHIII}. By forgetting the filtrations on $\mathbf{s}^r$
we recover the simple functor
$\mathbf{s}$ of complexes.

\begin{prop}\label{decala_simple_cxos}Let $K$ be a codiagram of filtered complexes.
Then for $r\geq 0$,
$$\Dec\left(\mathbf{s}^{r+1}(K)\right)\cong \mathbf{s}^r(\Dec K).$$
\end{prop}
\begin{proof}
The category $\FCx{\Aa}$ is complete. Furthermore, since
the d\'{e}calage has a left adjoint defined by the shift of a filtration (see \cite{CG2}), it commutes with pull-backs. 
It also commutes with $r$-translations $(K,W)\mapsto (K[r], W(-r))$. We have:
$$\Dec\int_\alpha (C^*(\Delta^{|\alpha|})\otimes W_{p+r|\alpha|}K^\alpha)\cong\int_\alpha \Dec(C^*(\Delta^{|\alpha|})\otimes W_{p+r|\alpha|}K^\alpha)\cong
\int_\alpha C^*(\Delta^{|\alpha|})\otimes \Dec W_pK^\alpha.$$
This gives an isomorphism of the filtrations $\Dec(W(r+1))$ and $(\Dec W)(r)$ on $\mathbf{s}(K)$.
\end{proof}

\begin{prop}\label{Ercommutasimple}Let $K$ be a codiagram of filtered complexes.
For $r\geq 0$, there is a chain of quasi-isomorphisms $E_r^{*,q}(\mathbf{s}^r(K))\stackrel{\sim}{\longleftrightarrow}\mathbf{s}E_r^{*,q}(K)$.
\end{prop}
\begin{proof}
For $r=0$ we have an isomorphism $E_0^{*,q}(\mathbf{s}^0(K))\cong\mathbf{s}E_0^{*,q}(K)$.
Assume inductively that the proposition is true for $r-1$. We then have a chain of quasi-isomorphisms
$$E_r^{*,q}(\mathbf{s}^r(K))\stackrel{\sim}{\leftarrow}E_{r-1}^{-q,*+2q}(\Dec(\mathbf{s}^r(K)))\cong
E_{r-1}^{-q,*+2q}(\mathbf{s}^{r-1}(\Dec K))\cong \mathbf{s} E_{r-1}^{-q,*+2q}(\Dec K)\stackrel{\sim}{\to} \mathbf{s}E_r^{*,q}(K),$$
where the first and last quasi-isomorphisms follow from Proposition $\ref{deligne_decalage}$ and the isomorphisms follow from 
Proposition $\ref{decala_simple_cxos}$ and the induction hypothesis
respectively.
\end{proof}

\begin{teo}\label{cohdescentfiltcomplexes}
Let $\Aa$ be an abelian category. The triple $(\FCx{\Aa}{},\Ee_r,\mathbf{s}^r)$ is a cohomological descent category for all $r\geq 0$.
\end{teo}
\begin{proof}
Consider the functor $E_r:\FCx{\Aa}\to \Cx{\Aa}$ defined by sending every filtered complex to the $r$-stage of its associated spectral sequence.
Then $\Ee_r=E_r^{-1}(\Ee)$, where $\Ee$ denotes the class of quasi-isomorphisms of $\Cx{\Aa}$.
Furthermore, by Proposition 
$\ref{Ercommutasimple}$, the complexes
$E_r(\mathbf{s}^r(K))$ and $\mathbf{s}E_r(K)$ are isomorphic in the derived category $\mathbf{D}^+(\Aa)=\Cx{\Aa}[\Ee^{-1}]$, for every codiagram $K$ in $\FCx{\Aa}$.
This isomorphism is compatible with the morphisms $\mu$ and $\lambda$ of Definition 1.5.3 of \cite{GN}.
By Proposition 1.7.2 of \cite{GN} the triple $(\Cx{\Aa},\Ee,\mathbf{s})$ is a cohomological descent category.
Hence by Proposition 1.5.12 of loc.cit., this lifts to a 
cohomological descent structure for the triple $(\FCx{\Aa},\Ee_r,\mathbf{s}^r)$.
\end{proof}

\begin{rmk}
For all $r\geq 0$, Deligne's d\'{e}calage is compatible with the cohomological descent structures 
$\Dec:(\FCx{\Aa},\Ee_{r+1},\mathbf{s}^{r+1})\to (\FCx{\Aa},\Ee_{r},\mathbf{s}^{r})$. Furthermore,
it induces an equivalence of categories
$\Dec:\mathbf{D}^+_{r+1}(\mathbf{F}\Aa){\lra}\mathbf{D}^+_r(\mathbf{F}\Aa)$
(see Theorem 2.19 of \cite{CG2}).
\end{rmk}

\section{Extension criterion of functors for analytic spaces}\label{ext_anal}
Let $\An{\CC}$ denote the category of complex analytic spaces that are reduced, separated and of finite dimension.
Denote by $\Man{\CC}$ the full subcategory of smooth manifolds.
\begin{defi}\label{defelement1}
A cartesian diagram of $\An{\CC}$
$$
\xymatrix{
\ar[d]_g\wt Y\ar[r]^j&\wt X\ar[d]^f\\
Y\ar[r]^i&X
}
$$
is said to be an \textit{acyclic square} if $i$ is a closed immersion, $f$ is proper and it induces an isomorphism
$\wt X-\wt Y\to X-Y$.
It is an \textit{elementary acyclic square} if, in addition, 
all the objects in the diagram are in $\Man{\CC}$, and $f$ is the blow-up of $X$ along $Y$.
In the latter case, the map $f$ is said to be an
\textit{elementary proper modification}.
\end{defi}

\begin{rmk}\label{finiteness}
In the analytic setting, we still have Hironaka's resolution of singularities.
However, in order to provide an extension criterion valid for analytic spaces,
we need to address certain issues concerning finiteness.

The first of this issues is Chow-Hironaka's Lemma (\cite{Hi}, 0.5),
stating that every proper birational map of irreducible schemes
factors as a composition of a finite sequence of blow-ups with smooth centers.
This result allows the passage
from acyclic squares to elementary acyclic squares in the hypotheses of the extension criterion.
In the analytic setting, the
factorization is made through the composition of a possibly infinite sequence of blow-ups, which is locally finite.
This is a consequence of
Hironaka's Flattening Theorem \cite{Hi2}.

The second issue concerns the finiteness of $\nu(X)=(n,c_n(X),\cdots,c_0(X))$,
where $c_i(X)$ is the number of irreducible components of dimension $i$, of a variety $X$ of dimension $n$,
which contain the singular points of $X$.
If $X$ is an algebraic variety, then $c_i(X)$ is finite for all $i$.
However, for an analytic space this may not be the case.
For compactificable analytic spaces, these two drawbacks disappear.
\end{rmk}

Denote by $\Manc{\CC}$ (resp. $\Anc{\CC}$) the full subcategory of $\Man{\CC}$ (resp. $\An{\CC}$) of compactificable analytic spaces.
The following is an analytic version of the extension criterion of functors defined over smooth schemes (see Theorem 2.1.10 of \cite{GN}).

\begin{teo}\label{descens1an}Let $\Dd$ be a cohomological descent category,
and let
$$F:\Manc{\CC}\lra \Ho(\Dd)$$
be a contravariant $\Phi$-rectified functor satisfying:
\begin{enumerate}
 \item [{(F1)}] $F(\emptyset)$ is the final object of $\Dd$ and $F(X\sqcup Y)\to F(X)\times F(Y)$ is an isomorphism.
 \item [{(F2)}] If $X_\bullet$ is an elementary acyclic square of $\Manc{\CC}$, then
$\mathbf{s}F(X_\bullet)$ is acyclic.
\end{enumerate}
Then there exists a contravariant $\Phi$-rectified
functor 
$$F':\Anc{\CC}\lra \Ho(\Dd)$$
such that:
\begin{enumerate}
\item [(1)]If $X$ is an object of $\Manc{\CC}$, then $F'(X)\cong F(X)$.
 \item [(2)]If $X_{\bullet}$ is an acyclic square of $\Anc{\CC}$, then $\mathbf{s}F'(X_{\bullet})$ is acyclic.
 \end{enumerate}
In addition, the functor $F'$ is essentially unique.
\end{teo}
\begin{proof}
With the notations of 2.1.10 of \cite{GN}, this is equivalent to prove that the inclusion functor
$\Manc{\CC}\to \Anc{\CC}$ verifies the extension property.
It suffices to replace $\Mm'=\Sm{\kk}$ by $\Manc{\CC}$ and $\Mm=\Sch{\kk}$ by $\Anc{\CC}$ in the proof of Theorem 2.1.5 of loc.cit.,
which by Remark $\ref{finiteness}$ is valid for compactificable analytic spaces.
\end{proof}

To prove the invariance of the weight filtration we will use a relative version of the above result.

Let $\An{\CC}^2$ denote the category of pairs $(X,U)$ where $X$ is an analytic space and
$U$ is an open subset of $X$ such that $D=X-U$ is a closed analytic subspace of $X$.

Likewise, let $\Man{\CC}^2$ be the full subcategory of $\An{\CC}^2$ of those pairs
$(X,U)$ with $X$ smooth and $D=X-U$ a normal crossings divisor in $X$ which is a union of smooth divisors.

\begin{defi}\label{defelement2}
A commutative diagram of $\An{\CC}^2$
$$
\xymatrix{
\ar[d]_g(\wt Y,\wt U\cap \wt Y)\ar[r]^j&(\wt X,\wt U)\ar[d]^{f}\\
(Y,U\cap Y)\ar[r]^i&(X,U)
}
$$
is said to be an \textit{acyclic square} if $f:\wt X\to X$ is proper, $i:Y\to X$ is a closed immersion,
the diagram of the first components is cartesian,
$f^{-1}(U)=\wt U$ and the diagram of the second components is an acyclic square of $\An{\CC}$.
\end{defi}

\begin{defi}
A morphism $f:(\wt X,\wt U)\to (X,U)$ in $\Man{\CC}^2$ is called \textit{proper elementary modification} if
$f:\wt X\to X$ is the blow-up of $X$ along a smooth center $Y$ which has normal crossings with the
 complementary of $U$ in $X$, and if $\wt U=f^{-1}(U)$.
\end{defi}

\begin{defi}
An acyclic square of objects of $\Man{\CC}^2$ is said to be an \textit{elementary acyclic square} 
if the map $f:(\wt X,\wt U)\to (X,U)$ is a proper elementary modification,
and the diagram of the second components is an elementary acyclic square of $\Man{\CC}$.
\end{defi}

Let $\An{\CC}^2_{comp}$ denote the full subcategory of $\An{\CC}^2$
given by those pairs $(X,U)$ such that $X$ is compact. Define $\Man{\CC}^2_{comp}$  similarly.
In particular, if $(X,U)\in \An{\CC}^2_{comp}$ we have that both $X$ and $U$ are objects of $\Anc{\CC}$.

\begin{nada}
Denote by $\gamma:\An{\CC}^2_{comp}\to \An{\CC}$ the forgetful functor $(X,U)\mapsto U$, and let $\Sigma$ be the class of morphisms
$s$ of $\An{\CC}^2_{comp}$ such that $\gamma(s)$ is an isomorphism. Then $\gamma$ induces a functor
$$\eta:\An{\CC}^2_{comp}[\Sigma^{-1}]\lra \An{\CC}.$$
In the algebraic situation, Nagata's Compactification Theorem implies that the functor $\eta^{alg}$
induces an equivalence of categories
$$\eta^{alg}:\Sch{\CC}^2_{comp}[\Sigma^{-1}]\stackrel{\sim}{\lra} \Sch{\CC}.$$
This does not hold in the analytic case. However,
the localized category $\An{\CC}^2_{comp}[\Sigma^{-1}]$ 
is equivalent to the category $\An{\CC}_\infty$ defined as follows.
\end{nada}

\begin{defi}[\cite{GN}, 4.7]\label{equivclass}
An \textit{object} of $\An{\CC}_\infty$ is given by an equivalence class of objects $(X,U)$ of $\An{\CC}^2_{comp}$,
where $(X,U)$ and $(X',U')$ are said to be \textit{equivalent} if
$U=U'$ and there exists a third compactification of $U$ that dominates $X$ and $X'$.

Two compactifications
$f_1:X_1\to X_1'$ and $f_2:X_2\to X_2'$ of a morphism of 
analytic spaces $f:U\to U'$ are said to be \textit{equivalent} 
if there exists a third compactification $f_3:X_3\to X_3'$ which dominates them.
A \textit{morphism} of $\An{\CC}_\infty$ is given by an equivalence class of morphisms of $\An{\CC}^2_{comp}$.
\end{defi}

An \textit{acyclic square} in $\An{\CC}_\infty$ is a square induced by an acyclic square of $\An{\CC}^2_{comp}$.
The following is an analytic version of Theorem 2.3.6 of \cite{GN}.
\begin{teo}\label{descens2an}Let $\Dd$ be a cohomological descent category,
and let
$$F:\Man{\CC}^2_{comp}\lra \Ho(\Dd)$$
be a contravariant $\Phi$-rectified functor satisfying:
\begin{enumerate}
 \item [{(F1)}] $F(\emptyset,\emptyset)$ is the final object of $\Dd$ and $F((X,U)\sqcup (Y,V))\to F(X,U)\times F(Y,V)$ is an isomorphism.
 \item [{(F2)}] If $(X_\bullet,U_\bullet)$ is an elementary acyclic square of $\Man{\CC}^2_{comp}$, then
$\mathbf{s}F(X_\bullet,U_\bullet)$ is acyclic.
\end{enumerate}
Then there exists a contravariant $\Phi$-rectified
functor 
$$F':\An{\CC}_\infty\lra \Ho(\Dd)$$
such that:
\begin{enumerate}
\item [(1)]If $(X,U)$ is an object of $\Man{\CC}^2_{comp}$, then $F'(U_\infty)\cong F(X,U)$.
 \item [(2)]If $U_{\infty\bullet}$ is an acyclic square of $\An{\CC}_{\infty}$, then $\mathbf{s}F'(U_{\infty\bullet})$ is acyclic.
 \end{enumerate}
In addition, the functor $F'$ is essentially unique.
\end{teo}
\begin{proof}If $(X,U)$ is an object of $\An{\CC}^2_{comp}$, then both $U$ and $X$ are objects of $\Anc{\CC}$.
Hence by Remark $\ref{finiteness}$, the proof of Theorem 2.3.3 of \cite{GN} 
applies to show that the inclusion functor
$\Man{\CC}^2_{comp}\to \An{\CC}^2_{comp}$ verifies the extension property 2.1.10 of loc. cit..
Therefore there exists a $\Phi$-rectified functor $F':\An{\CC}^2_{comp}\to \Ho(\Dd)$ satisfying:
\begin{enumerate}
\item [(1')]If $(X,U)$ is an object of $\Man{\CC}^2_{comp}$, then $F'(X,U)\cong F(X,U)$.
\item [(2')]If $(X_\bullet,U_\bullet)$ is an acyclic square of $\An{\CC}^2_{comp}$, then $\mathbf{s}F'(X_\bullet,U_\bullet)$ is acyclic.
 \end{enumerate}
Furthermore, the functor $F'$ is essentially unique. 
The remaining of the proof follows analogously to that of Theorem 2.3.6 of loc. cit., via
the equivalence of categories $$\eta:\An{\CC}^2_{comp}[\Sigma^{-1}]\stackrel{\sim}{\lra} \An{\CC}_\infty$$
given in 4.7 of loc.cit..
\end{proof}

\begin{rmk}\label{real}
The results of this section are also valid in the real setting, replacing $\CC$ by $\RR$, since Hironaka's resolution of singularities and 
the analytic version of
Chow-Hironaka's Lemma
hold for real analytic spaces (see \cite{Hi}).
\end{rmk}

\section{Acyclic squares and Gysin complex}\label{acyclicity}
In this section we study the behavior of the cohomology functor with respect to certain acyclic squares of smooth analytic spaces.
We then introduce the Gysin complex of a pair $(X,U)$, where $U\hookrightarrow X$ is a smooth compactification with $D=X-U$ a normal crossings divisor
and describe its behavior with respect to elementary acyclic squares.

Let $X$ be a complex analytic space.
Given a commutative ring $A$ we will denote by $\underline{A}_X$ the constant sheaf over $X$ associated to $A$
and by $H^*(X;A)$ the singular cohomology of $X$ with coefficients in $A$.
For a continuous map $f:X\to Y$ we will denote $\Rr f_*:=f_*\Cc^{\bullet}_{Gdm}$, where $\Cc^{\bullet}_{Gdm}$ is the Godement resolution.

\begin{prop}\label{element1}
For every acyclic square of $\Man{\CC}$ as in Definition $\ref{defelement1}$, the sequence
$$0\to H^q(X;A)\xra{(f^*,-i^*)} H^q(\wt X;A)\oplus H^q(Y;A)\xra{j^*+g^*} H^q(\wt Y;A)\to 0$$ 
is exact for all $q$.
\end{prop}
\begin{proof}
We have a Mayer-Vietoris long exact sequence
$$\cdots \to H^q(X;A)\to H^q(\wt X;A)\oplus H^q(Y;A)\to H^q(\wt Y;A)\to H^{q+1}(X;A)\to\cdots$$
Therefore it suffices to see that the map $f^*:H^q(X;A)\to H^q(\wt X;A)$ is injective.
This is a well known consequence of Poincar\'{e}-Verdier duality, which gives the existence of a trace morphism
$f_{\sharp}$ such that $f_{\sharp}f^*=1$.
We recall the proof.
The map $f^*$ is induced by a morphism of sheaves $\underline{A}_X\to \Rr f_*\underline{A}_{\wt X}$. Since $f$ is proper,
$\Rr f_*=\Rr f_{!}$, and we have an adjunction $\Rr f_*\vdash f^!$. Since $\wt X$ and $X$ are smooth and of the same pure dimension,
there is a quasi-isomorphism $f^!\underline{A}_X\stackrel{\sim}{\to}\underline{A}_{\wt X}$.
The trace map $f_\sharp:\Rr f_*\underline{A}_{\wt X}\to \underline{A}_X$ is deduced by adjunction from the identity morphism
$1_X\in Hom(\underline{A}_X,\underline{A}_X)$
of
$$\Hom(\Rr f_*\underline{A}_{\wt X},\underline{A}_X)\to \Hom(\underline{A}_{\wt X},f^!\underline{A}_X)\cong
\Hom(\underline{A}_{\wt X},\underline{A}_{\wt X}).$$
Lastly, since $f$ is birational, the composition $f_{\sharp}\circ f^*:\underline{A}_X\to\Rr f_*\underline{A}_{\wt X}\to \underline{A}_X$ 
is the identity, since it coincides with the identity
over an open dense subset of $X$. Hence $f_\sharp$ induces a left inverse
of $f^*$, and $f^*$ is injective.
\end{proof}

We shall also use the following blow-up formula for cohomology.
A proof can be found in Theorem VI.4.5 of \cite{FL}, which is an axiomatization 
of Theorem VII.3.7 of \cite{SGA6}.
\begin{prop}\label{blowup}
Consider an elementary acyclic square of $\Man{\CC}$ as in Definition $\ref{defelement1}$.
Let $m={\normalfont codim}_XY$, and let $\wt g^*=c_{m-1}(E)\cdot g^*:H^{*-2m}(Y;A)\to H^{*-2}(\wt Y;A) $, where
$c_{m-1}(E)\in H^{2m-2}(Y;A)$ denotes the $(m-1)$th-Chern class of the normal bundle $E=N_{Y/X}$ of $Y$ in $X$.
For all $q\geq 0$, there is a commutative square
$$
\xymatrix{
H^{q-2}(\wt Y;A)\ar[r]^{j_*}&H^q(\wt X;A)\\
\ar[u]^{\wt g^*}H^{q-2m}(Y;A)\ar[r]^{i_*}&H^q(X;A)\ar[u]_{f^*}
}
$$
such that the following sequence is exact
$$0\to H^{q-2m}(Y;A)\xra{(\wt g^*,-i_*)} H^{q-2}(\wt Y;A)\oplus H^q(X;A)\xra{j_*+f^*} H^q(\wt X;A)\to 0.$$ 
\end{prop}

\begin{nada}[Gysin complex]
Let $(X,U)\in \Man{\CC}^2$. 
We may write $D:=X-U=D_1\cup\cdots D_N$ as the union of irreducible smooth divisors meeting transversally.
Let $D^{(0)}=X$ and for $0<p\leq N$ let 
$D^{(p)}$ be the disjoint union of all $p$-fold intersections $D_I=D_{i_1}\cap\cdots\cap D_{i_p}$
with $I=\{i_1,\cdots,i_p\}\subset\{1,\cdots,N\}$. 
Since $D$ is a normal crossings divisor, $D^{(p)}$ is smooth. For all $q\geq 0$, the
\textit{Gysin complex} $G^q(X,U)$ is the cochain complex defined by
$$G^q(X,U)^p:=H^{q+2p}(D^{(-p)};A),$$
with $d^p:G^q(X,U)^p\to G^q(X,U)^{p+1}$ defined by the alternated sum of Gysin morphisms
$$i_{*}(I,J):H^{q+2p}(D_{J};A)\lra H^{q+2(p+1)}(D_{I};A),$$ where $I\subset J\subset\{1,\cdots,N\}$ and $|J|=|I|+1=-p$.
\end{nada}

\begin{lem}
For all $q\geq 0$, the Gysin complex defines a contravariant functor 
$$G^q:\Man{\CC}^2\to \Cx{\Rmod}.$$
\end{lem}
\begin{proof}
Let
$f:(X',U')\to (X,U)$ be a morphism
in $\Man{\CC}^2$.
Let $D=X-U=D_1\cup\cdots\cup D_N$ and $D'=X'-U'=D'_1\cup\cdots\cup D'_{M}$. For every irreducible component $D_i$ of $D$, its inverse image divisor is a sum
$$f^{-1}(D_i)=\sum_{j=1}^{M} m_{ij} D_j'$$ of irreducible components of $D'$. Let $M_f=(m_{ij})$ denote the matrix of multiplicities of $f$.
We next define $G^q(f)^*:G^q(X,U)^*\to G^q(X',U')^*$.
Let $G^q(f)^0=f^*$. Let $I\subset\{1,\cdots,N\}$ and $J\subset\{1,\cdots,M\}$ be two sets with $|I|=|J|=p>0$.
Let $m_{IJ}$ denote the determinant of the minor of $M_f$ of indices $(I,J)$.

If $I$ and $J$ are such that $f(D_J')\subset D_I$, we define a morphism
$G^q(f)_{IJ}:H^q(D_I)\to H^q(D'_J)$
by letting $G^q(f)_{IJ}:=m_{IJ}f_{IJ}^*$, where
$f_{IJ}:D'_{J}\to D_I$ denotes the restriction of $f$.
If $I$ and $J$ are such that $f(D_J')\nsubseteq D_I$, we let $G^q(f)_{IJ}=0$.
Then the morphisms $G^*_{IJ}(f)$ are the components of $G^q(f)^p:G^q(X,U)^p\to G^q(X',U')^p$.
It follows from the decomposition property of determinants, that this is a map of complexes (see \cite{GN}, pag. 84).

If $g:(X'',U'')\to (X',U')$ is a morphism of $\Man{\CC}^2$, 
then the matrix of multiplicities $M_{f\circ g}$ of $f\circ g$ is the product of the multiplicity matrices $M_f$ and $M_g$ of $f$ and $g$.
The functoriality of $G^q$ then follows from the functoriality of the determinants.
\end{proof}

\begin{prop}\label{element2}
Consider an elementary acyclic square of $\Man{\CC}^2$ as in Definition $\ref{defelement2}$.
\begin{enumerate}
 \item If $Y\nsubseteq D$ then the simple of the double  complex
 \begin{equation}0\to G^q(X,U)\to G^q(\wt X,\wt U)\oplus G^q(Y,U\cap Y)\to G^q(\wt Y,\wt U\cap \wt Y)\to 0
\tag{$\ast$}
\end{equation}
 is acyclic for all $q$.
 \item If $Y\subset D$ then the map $G^q(X,U)\to G^q(\wt X,\wt U)$ is a quasi-isomorphism for all $q$.
\end{enumerate}
\end{prop}
\begin{proof}
We adapt the proof of Proposition 5.9 of \cite{GN} in the motivic setting
(see also \cite{MCPII}, Sections 5 and 6).

Assume that $Y\nsubseteq D$. We proceed by induction on the number $N$ of smooth irreducible components of $D$.
If $N=0$ then $G^q(X,U)=H^q(X;A)$ is concentrated in degree 0 and the sequence ($\ast$) becomes
that of Proposition $\ref{element1}$.

Assume that $N>0$. Let $D=D''\cup X'$ and $D'=D''\cap X'$, where $X'$ is a component of $D$.
From the definition of the Gysin complex we obtain an exact sequence
$$0\to G(X,X-D'')\to G(X,X-D)\to G(X',X'-D')[1]\to 0.$$
Denote by $(X_\bullet,X_\bullet-D''_\bullet)$ the commutative square
$$
\xymatrix{
\ar[d](\wt Y,\wt Y- \wt E'')\ar[r]&(\wt X,\wt X-\wt D'')\ar[d]\\
(Y,Y- E'')\ar[r]^i&(X,X- D'')
}
$$
where $E''=D''\cap Y$, $\wt D''=f^{-1}(D'')$ and $\wt E''=\wt D''\cap \wt Y$. 
Consider the blow-up $\wt X'$ of $X'$ along $Y'=Y\cap X'$, and denote by
$(X_\bullet',X_\bullet'- D'_\bullet)$
the commutative square
$$
\xymatrix{
\ar[d](\wt Y',\wt Y'- \wt E')\ar[r]&(\wt X',\wt X'-\wt D')\ar[d]\\
(Y',Y'- E')\ar[r]^i&(X',X'- D')
}
$$
where
$E'=D'\cap Y'$, $\wt Y'=\wt X'\cap \wt Y$, $\wt D'=\wt X'\cap f^{-1}(D')$ and $\wt E'=\wt D'\cap \wt Y'$.
We then have a short exact sequence
$$0\to \mathbf{s}G(X_\bullet,X_\bullet- D''_\bullet)\to \mathbf{s}G(X_\bullet,X_\bullet- D_\bullet)\to 
\mathbf{s}G(X'_\bullet,X'_\bullet- D'_\bullet)[1]\to 0.$$
By induction hypothesis, both $\mathbf{s}G(X_\bullet,X_\bullet- D''_\bullet)$ and
$\mathbf{s}G(X'_\bullet,X'_\bullet- D'_\bullet)$ are acyclic complexes. 
Therefore the middle complex is acyclic, as desired. This proves (1).

Assume that $Y\subset D$. We proceed by induction over the number of components
$r$ of $D$ which contain $Y$, and the number $s$ of components
which do not contain $Y$.

Assume that $(r,s)=(1,0)$, so that $D$ is smooth irreducible and $Y\subset D$.
Then $G^q(X,X- D)$ is the simple of the morphism $H^{q-2}(D;A)\to H^q(X;A).$
Denote by $\widehat D$ the proper transform of $D$, and let $\widehat E=\wt Y\cap \widehat D$.
Denote by
$$\xymatrix{
\widehat E\ar[d]_{\widehat g}\ar[r]^{i_{\widehat E,\widehat D}}&\widehat{D}\ar[d]^{\widehat f}\\
Y\ar[r]^{i_{Y,D}}&D
}$$
the induced diagram.
Then
$\wt D=\wt Y\cup \widehat D$, and
$G^q(\wt X,\wt X- \wt D)$ is the simple of the square
$$
\xymatrix{
H^{q-4}(\widehat E;A)\ar[d]_{i_{\widehat E,\widetilde Y*}}\ar[r]^{i_{\widehat E,\widehat D*}}&H^{q-2}(\widehat D;A)\ar[d]^{i_{\widehat D,\widetilde X*}}\\
H^{q-2}(\wt Y;A)\ar[r]^{i_{\widetilde Y,\widetilde X*}}&H^{q}(\wt X;A)&.
}
$$
Therefore it suffices to show that the following complex is acyclic:
 $$H^{q-2}(D;A)\oplus H^{q-4}(\widehat E;A)\xra{\alpha} H^{q-2}(\widehat D;A)\oplus H^{q-2}(\wt Y;A)\oplus H^q(X;A)\xra{\beta} H^{q}(\wt X;A),$$
where  $\beta=i_{\widehat D,\widetilde X*}+i_{\widetilde Y,\widetilde X*}+f^*$ and $\alpha$ is given by the matrix
$$\left(
\begin{smallmatrix}
-\widehat f^*&-i_{\widehat E,\widehat D*}\\
-(i_{Y,D}\circ g)^*&i_{\widehat E,\widetilde Y*}\\
i_{DX*}&0
\end{smallmatrix}\right).
$$
After adding the acyclic complex $H^{q-2m}(Y;A)\lra H^{q-2m}(Y;A)$ and rearranging factors, we obtain the following complex
$$
\xymatrix@C=0pt@R=10pt{
& H^{q-2}(\wt Y;A) \ar[rrrrrrr]
& &&&&&& H^{q}(\wt X;A) \ar@{-}[dd]
\\
H^{q-2m}(Y;A) \ar[ur]\ar[rrrrrrr]
& &&&&&& H^{q}(X;A) \ar[ur]
\\
& H^{q-4}(\widehat E;A) \ar'[rrrrr]'[rrrrrr][rrrrrrr] \ar'[u][uu]
& &&&&&& H^{q-2}(\widehat D;A) \ar[uu]
\\
H^{q-2m}(Y;A) \ar[rrrrrrr]\ar[ur]\ar[uu]
&&&&&&& H^{q-2}(D;A) \ar[ur]\ar[uu] \ar'[uulllll][uuullllll]
}
$$
where the top and bottom faces of the cube are squares of blow-up type which are acyclic by Proposition $\ref{blowup}$.
Therefore the total complex of this complex is acyclic.
This proves (2) for the case $(r,s)=(1,0)$.

Assume that $r=1$ and $s>0$. Let $D=D''\cup X'$, where $Y\subsetneq X'$, and $D'=D''\cap X'$.
We have a commutative diagram with acyclic rows
$$
\xymatrix{
0\ar[r]&\ar[d]^{f^{''*}}G(X,X- D'')\ar[r]&\ar[d]^{f^{*}}G(X,X- D)\ar[r]&\ar[d]^{f^{'*}}G(X',X'- D')[1]\ar[r]&0\\
0\ar[r]&G(\wt X,\wt X- \wt D'')\ar[r]&G(\wt X,\wt X- \wt D)\ar[r]&G(\wt X',\wt X'- \wt D')[1]\ar[r]&0\\
}
$$
By induction hypothesis, the maps $f^{''*}$ and $f^{'*}$ are quasi-isomorphisms. Therefore $f^{*}$ is a quasi-isomorphism.

Assume that $r>1$ and consider a decomposition $D=D''\cup X'$ such that $Y\subset X'$.
An argument parallel to the previous case, by induction over $r$, shows that (2) is satisfied in the general case.
\end{proof}

\begin{rmk}\label{real_acyclic}
The results of this section have their analogues in the real setting, 
by taking cohomology with $\ZZ_2$-coefficients, for which Poincar\'{e}-Verdier duality holds.
Indeed, the same proof of Proposition $\ref{element1}$ can be carried out in this case.
Proposition $\ref{blowup}$ appears in Section 4 of \cite{MCPII} for $\ZZ_2$-homology.
Note that one needs to adjust the degree of the cohomology groups appearing in Proposition $\ref{blowup}$ and in the definition of the Gysin complex
as done in loc.cit., since
a closed immersion $f:X\hookrightarrow Y$ of real algebraic varieties
induces a Gysin map $f_*:H^k(X;A)\to H^{k+m}(Y;A)$, where $m=\normalfont{codim}_YX$.
The proof of Proposition $\ref{element2}$ follows analogously, with the Gysin complex defined by
$G^q(X,U)^p:=H^{q+p}(D^{(-p)};\ZZ_2)$.
\end{rmk}

\section{Singularity filtration}\label{sec_sing}
The singularity filtration is an analytic invariant that appears naturally when we extend the
functor of singular chains with the trivial filtration, from smooth to singular analytic spaces,
using the $\Ee_1$-cohomological descent structure on filtered complexes.

Let $X$ be a complex manifold.
Given a commutative ring $A$ denote by $S^*(X;A)$ the complex of singular cochains of $X$ with
values in $A$, so that $H^n(S^*(X;A))=H^n(X;A)$.
Together with the trivial filtration defined on $S^*(X;A)$ this defines a functor
$\Ss:\Man{\CC}\lra \FCx{\Rmod}$.
Any extended functor $\Anc{\CC}\to \Ho(\Dd)$ defined via Theorem $\ref{descens1an}$
depends strongly on the cohomological descent structure that we consider on the category $\Dd$.
In our case of interest, we may extend the functor $\Ss$ to a functor
$\Anc{\CC}\lra \mathbf{D}_0^+({\mathbf{F}\Rmod})$,
using the cohomological descent structure on $\FCx{\Rmod}$ associated with the class of $E_0$-quasi-isomorphisms.
It is easy to see that this is an empty exercise: the extended filtration of the trivial filtration is also trivial.
However, if we consider the cohomological descent structure associated with $E_1$-quasi-isomorphisms,
we obtain a non-trivial filtration which for compact spaces coincides with the weight filtration.

\begin{teo}\label{singan}
There exists a $\Phi$-rectified functor
$\Ss':\Anc{\CC}\lra \mathbf{D}^+_1(\mathbf{F}\Rmod)$
such that:
\begin{enumerate}[(1)]
\item If $X\in \Anc{\CC}$ then $H^n(\Ss'(X))\cong H^n(X;A)$.
\item If $X$ is a smooth manifold then $\Ss'(X)=(S^*(X;A), L)$, where $L$ is the trivial filtration.
 \item For every $p,q\in\ZZ$ and every acyclic square of $\Anc{\CC}$ as in Definition $\ref{defelement1}$
there is a long exact sequence
$$\cdots\to E_2^{p,q}(\Ss'(X))\to
E_2^{p,q}(\Ss'(\wt X))\oplus E_2^{p,q}(\Ss'(Y)) \to 
E_2^{p,q}(\Ss'(\wt Y))\to E_2^{p+1,q}(\Ss'(X))\to \cdots$$
 \item If $X$ is a compact complex algebraic variety and $A=\QQ$ then the filtration induced in cohomology
 coincides with Deligne's weight filtration after d\'{e}calage.
\end{enumerate}
\end{teo}
\begin{proof}
By Theorem $\ref{cohdescentfiltcomplexes}$ the triple $(\FCx{\Rmod}{},\Ee_1,\mathbf{s}^1)$
is a cohomological descent category. Therefore it suffices to show 
that the functor
$$\Manc{\CC}\stackrel{\Ss}{\lra}\FCx{\Rmod}\stackrel{\gamma}{\lra}\mathbf{D}^+_1(\mathbf{F}\Rmod)$$
given by $\Ss(X)=(S^*(X;A),t)$, where $t$ denotes the trivial filtration,
satisfies properties (F1) and (F2) of Theorem $\ref{descens1an}$.
Property (F1) is trivial. 

Let us prove (F2).
This is equivalent to the condition that
for every elementary acyclic square  
$X_\bullet\to X$ of $\Manc{\CC}$,
the map
$\Ss(X)\to \mathbf{s}^1(\Ss(X_\bullet))$ is an $E_1$-quasi-isomorphism.
By Proposition $\ref{Ercommutasimple}$, given a codiagram of filtered complexes $K^\bullet$, we have a chain of quasi-isomorphisms
$E_1^{*,q}(\mathbf{s}^1(K^\bullet))\stackrel{\sim}{\longleftrightarrow}\mathbf{s}E_1^{*,q}(K^\bullet).$
Hence
it suffices to check that for all $q\in\ZZ$, the sequence 
$$\cdots\to E_2^{*,q}(\Ss(X))\to
E_2^{*,q}(\Ss(\wt X))\oplus E_2^{*,q}(\Ss(Y)) \to E_2^{*,q}(\Ss(\wt Y))\to E_2^{*+1,q}(\Ss(X))\to \cdots
$$
is exact.
Since the filtrations are trivial,
$E_1^{0,q}(\Ss(-))=H^q(S^*(-;A))=H^q(-;A)$ and 
$E_1^{p,q}(\Ss(-))=0$ for $p\neq 0$.
Therefore it suffices to see that the sequence
$$0\to H^q(X;A)\lra H^q(\wt X;A)\oplus H^q(Y;A)\lra H^q(\wt Y;A)\to 0$$ 
is exact. This follows from Proposition $\ref{element1}$.
\end{proof}

\begin{defi}
Let $X$ be a compactificable complex analytic space.
The \textit{singularity spectral sequence} is the spectral sequence associated with the 
filtered complex $\Ss'(X)$ of Theorem $\ref{singan}$.
Let $L'$ denote the increasing filtration induced on $H^{*}(X;A)$. The
\textit{singularity filtration} $L_p$ on $H^{*}(X;A)$ is defined by $L_pH^{n}(X;A):=L_{p-n}'H^{n}(X;A)$.
\end{defi}

\begin{cor}
Let $X$ be a compactificable complex analytic space. Then for every $n\geq 0$, its cohomology 
$H^n(X;A)$ with values in any commutative ring $A$ carries a singularity filtration
$$0=L_{-1}\subset L_{0} \cdots\subset L_{n}=H^n(X;A)$$
which is functorial for morphisms in $\Anc{\CC}$ and satisfies:
\begin{enumerate}[1)]
\item If $X$ is smooth then $L$ is the trivial filtration.
\item If $X$ is a complex projective variety and $A=\QQ$ then $L$ coincides with Deligne's weight filtration.
\end{enumerate}
\end{cor}

Note that by Theorem $\ref{singan}$, the $_LE_2$-term of the singularity spectral sequence is well-defined.
The first term $_LE_1$, which is well-defined up to quasi-isomorphism, admits a description in terms of resolutions as follows:
let $X_\bullet\to X$ be a resolution of a compactificable complex analytic space $X$.
Then:
$$_LE_1^{p,q}(X;A)=\bigoplus_{|\alpha|=p} H^q(X_{\alpha};A) \Longrightarrow H^{p+q}(X;A).$$
If $X$ is a projective complex variety and $A=\QQ$ this corresponds to the analogous formula for the weight filtration
appearing in Theorem 8.1.15 of \cite{DeHIII} (see also IV.3 of \cite{GNPP}).

\begin{rmk}
The same arguments give a filtration $L$ on the homology with compact supports and on the Borel-Moore homology of a compactificable complex analytic space $X$.
In \cite{GZeeman}, Deligne's weight filtration $W$ and Zeeman's filtration $S$ are compared in the homology of a compact variety, giving the relation
$S^{2N-i-q}\subset W^{i-q}$ on $H_i(X;\QQ)$, where $N=\mathrm{dim}X$. The same proof would give the relation
$S^{2N-i-q}\subset L^{i-q}$ for the singularity filtration on the Borel-Moore homology $H_i^{BM}(X;A)$.
\end{rmk}

\begin{example}
Let $X$ be the open singular variety obtained by taking a Riemann surface of genus one,
identifying any two points and removing any two other points, as shown in the figure below.
The weight filtration on $H^1(X)\cong \QQ^4$ has length 3, and it is given by
$$Gr_0^WH^1(X)\cong \QQ[d],\, Gr_1^WH^1(X)\cong \QQ[a]\oplus\QQ[b]\text{ and }Gr_2^WH^1(X)\cong \QQ[c].$$
The singularity filtration has length 2, and it is given by
$$Gr_0^LH^1(X)\cong \QQ[d]\text{ and } Gr_1^LH^1(X)\cong \QQ[a]\oplus\QQ[b]\oplus\QQ[c].$$
\begin{center}
 \includegraphics[width=6cm]{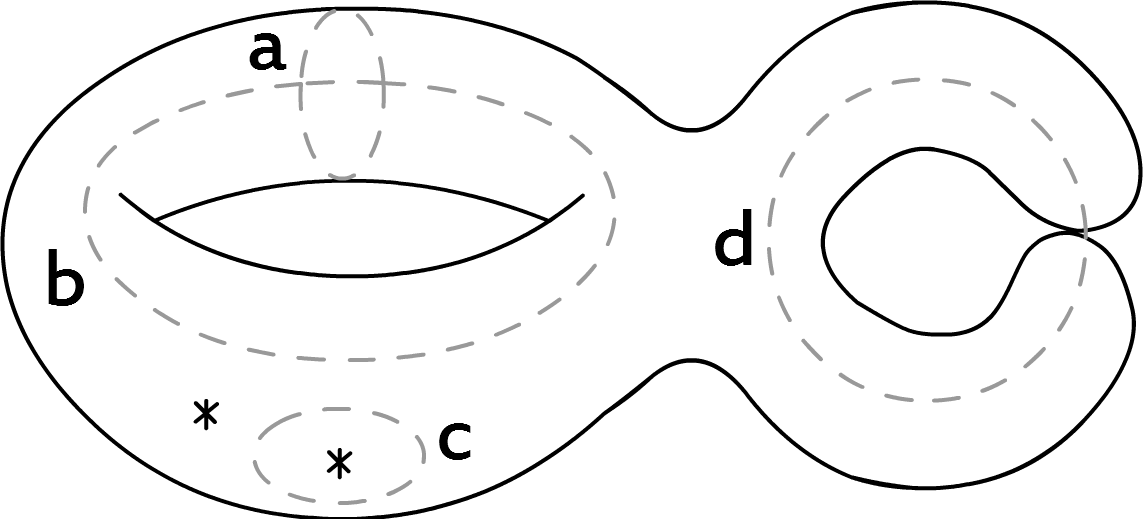}
\end{center}

\end{example}

Note that in the above example, the singularity filtration is simpler than the weight filtration, since it only captures the singularity of the variety.
In certain situations, the singularity filtration results in a finer invariant than the weight filtration,
since the contributions from the singular part and the part at the infinity may cancel out, as in the example below.

\begin{example}
Let $N$ denote the rational node, i.e. $\CC\PP^1$ with two points identified, and let $Y=N\times\CC^*$.
The weight and singularity filtrations on $H^1(Y)\cong \QQ^2$ have both length 2, and are given by $Gr_0^WH^1(Y)=Gr_0^LH^1(Y)\cong \QQ$ and
$Gr_2^WH^1(Y)=Gr_1^LH^1(Y)\cong \QQ$.
The weight filtration on $H^2(Y)\cong \QQ^2$ is pure of weight $2$, so 
$Gr_2^WH^2(Y)\cong \QQ^2$. In contrast, the singularity filtration
has two non-trivial graded pieces $Gr_1^LH^2(Y)\cong \QQ$ and $Gr_2^LH^2(Y)\cong \QQ$.
\end{example}

\section{Weight filtration}\label{sec_weightan}
Recall that the \textit{canonical filtration} $\tau$ is defined on any given complex $K$ by truncation:
$$\tau_{\leq p}K=\{\cdots\to K^{p-1}\to \Ker\,d\to 0\to 0\to\cdots\}.$$
Given $(X,U)\in \Man{\CC}^2_{comp}$, let $j:U\hookrightarrow X$ denote the inclusion,
and $(\Rr j_*\underline{A}_U,\tau)$ the filtered complex of sheaves given by the direct image of the constant sheaf $\underline{A}_U$,
together with the canonical filtration.
Taking the right derived functor of global sections we obtain a $\Phi$-rectified functor
$$\Ww:\Man{\CC}^2_{comp}\lra \mathbf{D}^+_1(\mathbf{F}\Rmod)$$
with values in the 1-derived category of filtered complexes of $A$-modules (see Definition $\ref{rderived}$),
given by 
$$\Ww(X,U)=\Rr\Gamma(X,(\Rr j_*\underline{A}_U,\tau)).$$
By the properties of the global sections functor and the derived direct image functor, we have an isomorphism $H^n(\Ww(X,U))\cong H^n(U;A)$.

\begin{teo}\label{weightan}
There exists a $\Phi$-rectified functor
${\Ww}':\An{\CC}_\infty\lra \mathbf{D}^+_1(\mathbf{F}\Rmod)$
such that:
\begin{enumerate}[(1)]
\item If $U_\infty\in \An{\CC}_\infty$ then $H^n({\Ww}'(U_\infty))\cong H^n(U;A)$.
\item If $(X,U)$ is an object of $\Man{\CC}^2_{comp}$ then ${\Ww}'(U_\infty)\cong \Ww(X,U)$.
 \item For every $p,q\in\ZZ$ and every acyclic square of $\An{\CC}_\infty$
$$
\xymatrix{
\ar[d]\wt Y_\infty\ar[r]&\wt X_\infty\ar[d]\\
Y_\infty\ar[r]&X_\infty
}
$$
there is a long exact sequence
$$\cdots\to E_2^{p,q}(\Ww'(X_\infty))\to
E_2^{p,q}(\Ww'(\wt X_\infty))\oplus E_2^{p,q}(\Ww'(Y_\infty)) \to 
E_2^{p,q}(\Ww'(\wt Y_\infty))\to E_2^{p+1,q}(\Ww'(X_\infty))\to \cdots$$
 \item If $X$ is a complex algebraic variety and $A=\QQ$ then the filtration induced in cohomology
 coincides with Deligne's weight filtration after d\'{e}calage.
\end{enumerate}
\end{teo}
\begin{proof}
By Theorem $\ref{cohdescentfiltcomplexes}$ the triple $(\FCx{\Rmod}{},\Ee_1,\mathbf{s}^1)$
is a cohomological descent category. Therefore by Theorem $\ref{descens2an}$ it suffices to show 
that the functor
$$\Ww:\Man{\CC}^2_{comp}\lra \mathbf{D}^+_1(\mathbf{F}\Rmod)$$
given by $\Ww(X,U)=\Rr\Gamma(X,(\Rr j_*\underline{A}_U,\tau))$
satisfies properties (F1) and (F2).
Property (F1) is trivial. 

Condition (F2)
is equivalent to the condition that
the map
${\Ww}(X,U)\to \mathbf{s}^1(\Ww(X_\alpha,U_\alpha))$ is an $E_1$-quasi-isomorphism
for every elementary acyclic square  
$(X_\alpha,U_\alpha)\to (X,U)$ of $\Man{\CC}^2_{comp}$.

As in the proof of Theorem $\ref{singan}$, it suffices to check that
for all $q\in\ZZ$, the sequence 
$$\cdots\to E_2^{*,q}({\Ww}(X,U))\to
E_2^{*,q}({\Ww}(\wt X,\wt U))\oplus E_2^{*,q}({\Ww}(Y,U\cap Y)) \to E_2^{*,q}({\Ww}(\wt Y,\wt U\cap\wt Y))\to \cdots$$
is exact.
Since $E_1^{*,q}({\Ww}(X,U))$ is the shifted Leray spectral sequence of the inclusion $j:U\hookrightarrow X$,
it is isomorphic to the Gysin complex
$G^q(X,U)^*$ for all $q$ (see \cite{DeHII}, 3.1.9 and 3.2.4, see also Section 4.3 of \cite{PS}).
Hence the exactness of this sequence follows from Proposition $\ref{element2}$.
\end{proof}

\begin{defi}
Let $X_\infty$ be a complex analytic space with an equivalence class of compactifications.
The \textit{weight spectral sequence} is the spectral sequence associated with the 
filtered complex $\Ww'(X_\infty)$.
If $W'$ denotes the induced filtration on $H^{*}(X;A)$, the \textit{weight filtration} $W_p$
on $H^{*}(X;A)$ is defined by $W_pH^{n}(X;A):=W_{p-n}'H^{n}(X;A)$.
\end{defi}

\begin{cor}
Let $X_\infty$ be a complex analytic space with an equivalence class of compactifications. For every $n\geq 0$, its cohomology 
$H^n(X;A)$ with values in any commutative ring $A$ carries a weight filtration
$$0=W_{-1}\subset W_{0} \cdots\subset W_{2 n}=H^n(X;A)$$
which is functorial for morphisms in $\An{\CC}_\infty$ and satisfies:
\begin{enumerate}[1)]
\item If $X$ is smooth then $0=W_{n-1}\subset H^n(X;A)$.
\item If $X$ is compact then $W_{n}=H^n(X;A)$.
\item If $X$ is a complex algebraic variety and $A=\QQ$ then $W$ is Deligne's weight filtration.
\end{enumerate}
\end{cor}

Note that by Theorem $\ref{weightan}$ the weight spectral sequence of $X_\infty$ is well-defined 
 from the $E_2$-term onwards. The first term $_WE_1$ of the weight spectral sequence admits a description in terms of compactifications and resolutions as follows:

1) Assume that $X_\infty$ is smooth. Choose a representative $X\hookrightarrow \overline{X}$ of the compactification 
class $X_\infty$ with $D=\overline{X}-X$
a normal crossings divisor. Denote by $D^{(p)}$ the disjoint union of all $p$-fold intersections of the smooth irreducible components of $D$.
Then:
$$_WE_1^{-p,q}(X_\infty;A)=H^{q-2p}(D^{(p)};A)\Longrightarrow H^{q-p}(X;A).$$
If $X$ is algebraic and $A=\QQ$ we recover Deligne's formula 3.2.4 of \cite{DeHII}.

2) Assume that $X$ is compact. Let $X_\bullet\to X$ be a cubical hyperresolution of $X$.
Then:
$$_WE_1^{p,q}(X;A)=\bigoplus_{|\alpha|=p}H^q(X_{\alpha};A)\Longrightarrow H^{p+q}(X;A).$$

3) For the general case, let $X\hookrightarrow \overline{X}$ be a representative of $X_\infty$,
and let $(\overline{X}_\bullet,X_\bullet)\to (\overline{X},X)$ be a resolution of $(\overline{X},X)$.
These are resolutions $X_\bullet\to X$ and $\overline{X}_\bullet\to \overline{X}$ such that
the complement $D_\alpha=\overline{X}_\alpha- X_\alpha$ is a normal crossings divisor for each $\alpha$.
Then:
$$_WE_1^{-p,q}(X_\infty;A)=\bigoplus_{\alpha}{\color{white}.}_WE_1^{-p-|\alpha|,q}(X_{\alpha};A)=\bigoplus_{\alpha}H^{q-2p-2|\alpha|}(D^{(p+|\alpha|)}_\alpha;A) \Longrightarrow H^{q-p}(X;A).$$
If $X$ is algebraic and $A=\QQ$ this corresponds to the analogous formula appearing in Theorem 8.1.15 of \cite{DeHIII} (see also IV.3 of \cite{GNPP}).

\begin{rmk}In general, the weight filtration $W$ on the cohomology of a compactificable complex analytic space depends on the class of its compactification,
and it is functorial for morphisms in $\An{\CC}_\infty$ (see Example $\ref{compadep}$ below).
If $X$ is a complex algebraic variety, by Nagata's Theorem, it admits a unique equivalence class of compactifications.
Therefore we obtain a weight filtration on $H^*(X;A)$ for an arbitrary coefficient ring $A$,
which is independent of the compactification and is functorial for morphisms of algebraic varieties.
\end{rmk}

\begin{example}\label{compadep}
Serre constructed two non-equivalent analytic compactifications on $U=\CC^*\times\CC^*$
(see \cite{Har}, Appendix B, 2.0.1).
The first is $\CC\PP^1\times\CC\PP^1$, and gives a weight filtration of pure weight 2 on $H^1(U;\QQ)$.
The second compactification is defined as the total space of a certain $\PP^1$-bundle, and gives a weight filtration of pure weight 1 on $H^1(U;\QQ)$
(see Example 4.19 of \cite{PS} for details).
Hence $U^{an}$ is an example of a complex analytic space having two different classes of compactifications, which give different weight filtrations.
\end{example}

\begin{rmk}The techniques used previously are also applicable to real analytic spaces, leading to the existence of a functorial
weight filtration on the $\ZZ_2$-cohomology of every real analytic space with an equivalence class of compactifications. Indeed, 
by Remarks $\ref{real}$ and $\ref{real_acyclic}$, the extension criterion and the results of Section $\ref{acyclicity}$ are valid 
in the real setting. A similar proof of Theorem $\ref{weightan}$ can be carried out in this case, extending the functor
$\Ww:\Man{\RR}^2_{comp}\lra \mathbf{D}^+_1(\mathbf{F}\ZZ_2-\text{mod})$
given by $\Ww(X,U)=\mathbf{s}G^*(X,U)^*$, with the filtration by the second degree.
\end{rmk}

\begin{example}[cf. \cite{MCPII}]
Let $U=\RR^2-\{*\}$ be the punctured plane and 
$V=\RR\PP^1\times\RR$ the infinite cylinder.
The weight filtration on $H^1(U;\ZZ_2)$ is pure of weight 2, while on $H^1(V;\ZZ_2)$ it is pure of weight 1.
Note that while $U$ and $V$ are not isomorphic algebraic varieties, they are isomorphic as real analytic spaces.
Hence $U^{an}$ is an example of a real analytic space having two different classes of compactifications 
which give different weight filtrations,
corresponding to $\overline{U}=\widetilde {\RR\PP}^2$ and $\overline{V}=\RR\PP^1\times \RR\PP^1$,
where $\widetilde {\RR\PP}^2$ is the blow-up of $\RR\PP^2$ at a point.
\end{example}

\section*{Acknowledgments}
The results of this paper are based on a non-published manuscript by the second named author together with V. Navarro-Aznar.
We deeply thank him for his generosity and useful comments.

\bibliographystyle{amsplain}
\bibliography{bibliografia}

\end{document}